\documentclass[11pt,a4paper]{article}
\usepackage{geometry} 
\usepackage[utf8]{inputenc}
\usepackage{amsmath}
\usepackage{amsthm}
\usepackage{amssymb}
\usepackage{enumitem}

\newtheorem{theorem}{Theorem}[section]
\newtheorem{lemma}[theorem]{Lemma}
\newtheorem{corollary}[theorem]{Corollary}

\newtheorem{Example}[theorem]{Example}

\theoremstyle{definition}
\newtheorem{remark}[theorem]{Remark}

\title{Asymptotic zero distribution of random orthogonal polynomials}
\author{Thomas Bloom and Duncan Dauvergne}

\begin{document}
\maketitle
\begin{abstract}
We consider random polynomials of the form $H_n(z)=\sum_{j=0}^n\xi_jq_j(z)$ where the $\{\xi_j\}$
are i.i.d non-degenerate complex random variables, and the $\{q_j(z)\}$ are orthonormal polynomials with respect to a compactly supported measure $\tau$ satisfying the Bernstein-Markov property on a regular compact set $K \subset \mathbb{C}$. We show that if $\mathbb{P}(|\xi_0|>e^{|z|})=o(|z|^{-1})$, then the normalized counting measure of the zeros of $H_n$ converges weakly in probability to the equilibrium measure of $K.$ This is the best possible result, in the sense that the roots of $G_n(z)=\sum_{j=0}^n\xi_jz^j$ fail to converge in probability to the appropriate equilibrium measure when the above condition on the $\xi_j$ is not satisfied. In addition, we  give a multivariable version of this result.

We also consider random polynomials of the form $\sum_{k=0}^n\xi_kf_{n,k}z^k$, where the coefficients $f_{n,k}$ are complex constants satisfying certain conditions, and the random variables $\{\xi_k\}$ satisfy $\mathbb{E} \log(1 + |\xi_0|) < \infty$.  In this case, we establish almost sure convergence of the normalized counting measure of the zeros to an appropriate limiting measure. Again, this is the best possible result in the same sense as above.
 
\end{abstract}
\section{Introduction}
In this paper we will be concerned with the global distribution of the complex zeros of random polynomials in both one and several variables.

\medskip

The origin of the problems goes back to results on the \textit{Kac ensemble} of random polynomials
$$H_n(z)=\sum_{j=0}^n\xi_jz^j,$$
where the $\xi_j$ are i.i.d. non-degenerate complex-valued random variables. Here a random-variable is non-degenerate if its law is supported on at least two points. The interest is in the behaviour of the zeros of $H_n(z)$ as $n\rightarrow\infty.$ 

\medskip

In the case that the $\xi_j$ 
are i.i.d. complex Gaussians of mean zero and variance one, it is a classical result of Hammersley \cite{Ha} that the zeros tend to concentrate on the unit circle $\{|z|=1\}.$ 

\medskip

The Kac ensemble has been extensively studied (see \cite{HN,IZ,SV}). In particular, Ibragimov and Zaporozhets \cite{IZ} showed that the condition
\begin{equation}
\label{E:log-exp}
\mathbb{E}(\log(1+|\xi_0|))<\infty
\end{equation}
is both necessary and sufficient for almost sure weak* convergence of the normalized counting measure of the zeros (i.e. $\frac{1}{n}\sum_{j=1}^n\delta(z_j)$ where $z_1,\ldots ,z_n$ are the zeros of $H_n$) to normalized Lebesgue measure on the unit circle, $\frac{1}{2\pi}d\theta$.

\medskip

Shiffman and Zelditch \cite{SZ} took the point of view that the functions
$\{z^i\}_{i=0,1,\cdots}$ are an orthonormal basis for the polynomials in $L^2(\frac{1}{2\pi}d\theta)$. Generalizing this idea, they considered random polynomials of the form
\begin{equation}
\label{E:orthog-ensemble}
H_n(z)=\sum_{j=0}^n\xi_jq_j(z),
\end{equation}
where the $\xi_j$ are complex Gaussians of  mean zero and variance one, and the $q_j$ are an orthonormal basis for the polynomials in $L^2(d\mu)$ for certain measures  $\mu$ with compact support $K$ in the complex plane. They showed that the normalized counting measure of the zeros converges almost surely to the equilibrium measure of $K$ in the weak* topology. Various generalizations to the several variables situation are in \cite{B1},\cite{BL}, and \cite{BS}. The papers \cite{B1} and \cite{BL} include results in the case that the i.i.d. coefficients are more general than Gaussians.

\medskip

The problems studied in this paper have also been studied in other contexts. For random polynomials using a basis other than orthogonal polynomials, see \cite{PR}. For random holomorphic sections of a line bundle
see \cite{BCM}, \cite{SZ1}.

\medskip

In this paper we will primarily be concerned with finding the weakest possible conditions on the i.i.d. coefficients so that we still obtain the same limiting behaviour of the zeros.
In \cite{Ba1}, Bayraktar established the almost sure weak* convergence of the zeros of random polynomials of the form \eqref{E:orthog-ensemble}
for i.i.d. coefficients $\xi_i$ with a continuous Lebesgue density, satisfying
$$
\mathbb{P}(|\xi_0|\geq e^{|z|})=O(|z|^{-\rho}),
$$
where polynomials are considered in $d$ variables and $\rho>d+1.$  His results included the case of common zeros of $m$-tuples of random polynomials where $m\leq d.$

\medskip

In this paper, in Theorem \ref{goal} for the one-variable case and Theorem \ref{goalmulti} for the several variable case, we establish convergence in probability for the zeros of random polynomials with i.i.d coefficients satisfying
\begin{equation}
\label{E:prob-cond9}
\mathbb{P}(|\xi|\geq e^{|z|})=o(|z|^{-d})
\end{equation}
for polynomials in $d$ variables. We require no assumption about whether the $\xi_i$ have a Lebesgue density. This is a best possible result in the sense that if \eqref{E:prob-cond9} is not satisfied for the Kac ensemble, then the normalized counting measure of the zeros does not converge in probability. We prove this in Theorem \ref{T:suff-Kac}. We do not deal with the case of common zeros of random polynomials in several variables in this paper.

\medskip

As a corollary of Theorem \ref{goal}, we resolve a conjecture of Pristker and Ramachandran (\cite{PR}, Conjecture 2.5). Very roughly, they asked if there exist i.i.d. random variables $\{\xi_i\}_{i \in \mathbb{N}}$ and a sequence $\{q_i\}_{i \in \mathbb{N}}$ of orthonormal polynomials with respect to a measure $\tau$ on the unit circle
such that the normalized counting measures of the zeros of 
$$
H_n(z) = \sum_{i = 0}^n \xi_i q_i(z)
$$
converge almost surely along a subsequence $\{n_i\}$, but not almost surely along the whole sequence. In Remark \ref{R:pritsker}, we use Theorem \ref{goal} to construct random polynomials $H_n(z)$ with this property.

\medskip

The basic strategy of the proofs is to prove convergence of the normalized logarithmic potential 
\begin{equation}
\label{E:H_n-log}
\frac{1}n \log |H_n(z)|
\end{equation}
to the Green function $V_K$ of the compact set $K.$ In the one variable case this implies that the normalized counting measure of the zeros of $H_n(z)$ converges to the equilibrium measure in the weak* topology. In the several variable case, if we consider the zero set as a $(1,1)$ current, then this current  converges to the canonical $(1,1)$ current $dd^cV_K.$

\medskip

The main difficulty here is in establishing lower bounds on the function in \eqref{E:H_n-log}. This is accomplished in Theorem \ref{convprob} for the one variable case. The multivariable case is analogously done in Theorem \ref{convprobmulti}. To do this we use the Kolmogorov-Rogozin inequality. This inequality has been previously used to establish lower bounds of this type in \cite{KZ}. Unlike in that paper, our arguments dispense with the need for circular symmetry of the polynomials when applying the inequality.
 
\bigskip

In addition to the above results on convergence in probability we also consider almost sure convergence of the normalized zero counting measure of random polynomials of the form
\begin{equation}
\label{E:random-analytic}
G_n(z) = \sum_{i=0}^n \xi_i f_{n, i} z^i.
\end{equation}
Here the $\xi_i$ are non-degenerate i.i.d. complex random variables satisfying
\begin{equation}
\label{E:as-cond}
\quad \quad \mathbb{E}\log(1 + |\xi_0|) < \infty \qquad \text{(or $\mathbb{E}[\log(1 + |\xi_0|)]^d < \infty$ in $d$ variables)}
\end{equation}
 and the coefficients $\{ f_{n, i} :0 \le i \le n, n \in \mathbb{N}\}$ satisfy
$$
\lim_{n \to \infty} \frac{1}n \log \left( \sum_{k=0}^n |f_{n, k}|r^k \right) = V(r).
$$
Here $V(r)$ is a continuous function and the convergence above is locally uniform.
We will also assume that sufficiently many of the coefficients $f_{n, k}$ are large enough. This assumption will be made precise in Section \ref{S:almost-sure}. 
 
 \medskip

The conditions on the coefficients $f_{n, k}$ are quite general, and therefore the ensembles of the form \eqref{E:random-analytic} include many examples of random polynomials. For example, random polynomials of the form \eqref{E:orthog-ensemble}, where the measure $\mu$ is rotationally symmetric, satisfy these conditions. Random polynomials formed from an array of orthogonal polynomials induced by a rotationally symmetric measure and a rotationally symmetric weight function also fit into this category (see Section \ref{S:almost-sure} for details).

\medskip

Random polynomials with slightly stronger restrictions on the coefficients $f_{n, k}$ were analyzed by 
Kabluchko and Zaporozhets in \cite{KZ}. In that paper, the authors proved that the normalized counting measure of the zeros of $G_n(z)$ converges in probability to the appropriate limiting measure, and asked when this could be extended to almost sure convergence. They showed almost sure convergence for a few particular arrays using ad hoc methods.

\medskip

In Theorem \ref{T:general-KZ-roots}, we prove almost sure convergence of the normalized counting measure of the zeroes for random polynomials of the form \eqref{E:random-analytic} with the above conditions imposed on the $f_{n, k}$s, answering the question of Kabluchko and Zaporozhets. 

\medskip

 The roots converge to a measure $\nu$ that is equal as a distribution to $\frac{1}{2\pi} \Delta V(|z|)$. This theorem can also be extended to the multivariable case. We prove one particular example of this, the two-variable Kac ensemble, in Theorem \ref{T:2-Kac}. 
 
 \medskip

Note that again, the condition $\mathbb{E} \log(1 + |\xi_0|) < \infty$ on the random variables is the best possible, in the sense that if this condition fails, then almost sure convergence fails for the Kac ensemble (see \cite{IZ} for details).

\medskip

Kabluchko and Zaporozhets also considered random analytic functions of a similar form. Our methods can be easily extended to include this case, but we choose to only address random polynomials in this paper for ease of exposition.

\medskip

Again, our method for almost sure convergence is based on proving convergence of the function in \eqref{E:H_n-log}. The main obstacle is again in obtaining a lower bound on $\frac{1}n \log |G_n(z)|$. For proving almost sure convergence the Kolmogorov-Rogozin inequality is too weak, so a stronger concentration inequality is needed.

\medskip

 For this, we use a small ball probability theorem of Nguyen and Vu \cite{NV}. This gives a stronger concentration estimate than the Kolmogorov-Rogozin inequality for sums of the form $\sum_{i=1}^n \xi_i a_i$, where the $a_i$s are fixed and the $\xi_i$s are i.i.d. random variables. This stronger estimate requires that the coefficients $a_i$ are sufficiently spread out in the plane.

 \bigskip
 
 \noindent \textbf{Further work.} \qquad In a follow-up paper \cite{D}, the second author uses the small ball probability techniques of Section \ref{S:almost-sure} to prove almost sure convergence of the normalized counting measure of the zeros for general random orthogonal polynomials of the form \eqref{E:orthog-ensemble} under the condition \eqref{E:as-cond}. Necessity of this condition (as well the condition $\mathbb{P}(|\xi| > e^{|z|}) = o(|z|^{-1})$ for convergence in probability) is also proven in \cite{D}.
 
\section{Preliminaries}
In this section we recall some basic results in potential theory.

\medskip

\noindent Let $D\subset\mathbb{C}$ be an open set. A function $u$ on $D$ is \textit{subharmonic} if it
\begin{enumerate}[nosep, label = (\roman*)]
\smallskip

\item takes values in $[-\infty,+\infty)$.
\item is upper semicontinuous.
\item satisfies the submean inequality. That is, given $w\in D$, there exists $\rho>0$ such that
$$
u(w)\leq\frac{1}{2\pi}\int^{2\pi}_0u(w+re^{it})dt\quad (0\leq r<\rho).
$$
\end{enumerate}
We denote by $sh(D)$ the collection of subharmonic functions on $D$. If $f$ is analytic on $D$ then $\log |f|\in sh(D)$.

\medskip

A set $E\subset \mathbb{C}$ is \textit{polar} if there is a non-constant subharmonic function $u$ with $E\subset \{u=-\infty\}$. Subharmonic functions are locally integrable  and thus polar sets are of Lebesgue planar measure zero.

\medskip

For a function $f$ on $D$ we denote by $f^*$ its uppersemicontinuous regularization given by
$$
f^*(z):=\limsup_{w\rightarrow z}f(w).
$$
Let $\mathcal{P}_n$ denote the space of polynomials of degree $\leq n$. Let 
$$
\mathcal{L}(\mathbb{C})=\big\{u\in sh(\mathbb{C}) \big|\quad u(z)-\log |z| \quad\text {is bounded above as}\quad |z| \rightarrow +\infty \big\}.
$$
If $p\in\mathcal{P}_n$, then $\frac{1}{deg(p)}\log|p|\in\mathcal{L}(\mathbb{C}).$

\medskip

 For $K\subset\mathbb{C}$, the Green function of $K$ is given by
\begin{equation}
\label{E:V_K}
V_K(z):=\sup \left\{\frac{1}{deg(p)}\log |p(z)| \; \Big| \; p \;\; \text{is a non-constant polynomial, and}\quad ||p||_K \leq 1 \right\}.
\end{equation}
Whenever $K$ is non-polar, $V_K^*\in \mathcal{L}(\mathbb{C}).$ Note that 
$V_K=V_{\tilde{K}}$ where 
$$
\tilde{K} :=\{z:|p(z)|\leq ||p||_K \;\;\text{for all polynomials}\;\; p\}
$$
is the \textit{polynomially convex hull} of $K$. 

\medskip

We say that $K$ is \textit{regular} when $V_K$ is continuous, i.e. $V_K=V_K^*.$ This is equivalent to the unbounded component of the complement of $K$ being regular for the Dirichlet problem. For $K$ regular, $V_K=0$ for all $z\in \tilde{K}$. In this paper we will restrict to regular sets. $V_K$ is harmonic on $\mathbb{C}\setminus\tilde{K}$ and  has the asymptotic expansion
\begin{equation*}
V_K(z)=\log|z| -\log c(K)+o(1)
\end{equation*}
as $|z|\rightarrow \infty.$ Here $c(K)$ denotes the \text{logarithmic capacity} of $K.$

\begin{Example}
\label{Ex:kac}
Let $K=\{z:|z|=1\}$ be the unit circle in the plane. Then $K$ is regular and $V_K(z) =\max (0,\log |z|)$. The polynomially convex hull of $K$ is given by $\tilde{K}=\{z:|z|\leq 1\}.$
\end{Example}

The following theorem is from \cite{R} (Theorems 3.4.2 and 3.4.3).
\begin{theorem}\label{a}
Let $D\subset\mathbb{C}$ be open. Let $\{\psi_n(z)\}_{n=1,2,...}$ be a sequence in $sh(D)$ which is locally bounded above. Then 
$$ w(z):=(\limsup_n\psi_n(z))^* $$ and $$w_1(z):=(\sup_n\psi_n(z))^*$$
are subharmonic on $D$. Furthermore,
$w(z)=\limsup_n\psi_n(z)$ outside of a polar set,
and $w_1(z)=\sup_n\psi_n(z)$ outside a polar set.
\end{theorem}

We present the following simple result without proof:
\begin{lemma}\label{usc}
 Let $f$ be upper semicontinuous and $g$ continuous on $D$ with $f\leq g.$ Suppose that $f=g$ at a dense set of points in $D.$ Then $f=g$ on $D.$
\end{lemma}

Let $L^1_{loc}(D)$ denote the space of locally integrable functions on $D$. The next theorem gives conditions for a sequence of subharmonic functions to converge in $L^1_{loc}(D)$.

\begin{theorem}
\label{b}
(see also \cite{BL}, Proposition 4.4).
Let $D\subset\mathbb{C}$ be open. Let $\{\psi_n(z)\}_{n=1,2,\dots}$ be a sequence in $sh(D)$ which is locally bounded above and let $w(z)\geq (\limsup_n\psi_n(z))^*$.
Suppose that $w$ is continuous and that there is a countable dense set of points \\
$\{z_i\}_{i\in \mathbb{N}} \subset D$ such that 
$$\lim_{n \to \infty} \psi_n(z_i)=w(z_i) \quad \text{ for all } i \in \mathbb{N}.$$
Then $\psi_n \rightarrow w$
in $L^1_{loc}(D)$.
\end{theorem}
\begin{proof}
The first step is to show that for any subsequence $J\subset\mathbb{N}$ and any $z \in D$, that 
$$
(\limsup_{n\in J}\psi_n(z))^*=w(z).
$$
To this end let $J\subset\mathbb{N}$ 
be a subsequence.
By Theorem \ref{a},
$$
w_J:=(\limsup_{n\in J}\psi_n(z))^*
$$
is subharmonic on D. By Lemma \ref{usc}, $w_J=w$ on $D.$ This completes the first step.

\medskip

   Next we proceed by contradiction to prove the theorem. Suppose that the conclusion of the theorem does not hold. Then there exists a closed ball $B\subset D$ and $\epsilon> 0$ such that for some subsequence $J_1\subset\mathbb{N}$ we have, for $n\in J_1$,
   
   \begin{equation}\label{nonconv}
   ||\psi_n-w||_{L^1(B)}\geq\epsilon.
   \end{equation}
   
    However, appealing to Theorem 3.2.12 of \cite{H}, there is a subsequence $J_2\subset J_1$ and $g\in L^1(B)$ with $\lim_{n\in J_2}\psi_n=g $ in $L^1(B)$. It follows from standard measure theory that there is a further subsequence $J_3\subset J_2$ with $\lim_{n\in J_3}\psi_n(z) =g(z)$ for a.e. $z\in B$  so that $g(z) = w_{J}(z) = w(z)$ a.e. in $B$. This contradicts (\ref{nonconv}).
\end{proof}

We remark that $L^1_{loc}(D)$ may be endowed with a metric as follows:

\begin{remark}
Let $L^1_{loc}(D)$ denote the space of functions locally in $L^1$ on an open set $D\subset\mathbb{C}.$ The space $L^1_{loc}(D)$ is a metric space as follows: let $X_1,X_2,...$ be a sequence of compact subsets of $D$ with $\cup_{i=1}^{\infty}X_i=D, X_i\subset X_{i+1}$ for all $i$. For $f,g \in L^1_{loc}(D)$ set
$$
\rho(f,g):=\sum_{i=1}^{\infty}2^{-i}\min[1,||f-g||_{L^1(X_i)}].
$$

\end{remark}
\section{Construction of Random Polynomials}
\label{S:construct}

We will construct random polynomials by ``randomizing" linear combinations of orthogonal polynomials. 
We consider random polynomials of the form
\begin{equation}\label{r1}
H_n(z):=\sum_{j=0}^n\xi_jq_j(z)
\end{equation}
where the $\xi_j$ are i.i.d complex-valued random variables, and the $\{q_j(z)\}$ are orthonormal polynomials constructed below.

\medskip

To emphasize the randomness we will sometimes use the notation
$H_n(z,\omega)$ where $\omega\in\Omega$ and the i.i.d random variables $\xi_i$ are defined on a probability space $(\Omega,\mathcal{F},\mathbb{P})$.

\medskip

 Let $K$ be a compact and regular subset of $\mathbb{C}.$ We construct the polynomials $q_j(z)$ as follows:
 
 \medskip

Let $\tau$ be a finite measure on $K.$ Apply the Gram-Schmidt orthogonalization procedure to the monomials $\{1,z,\dots,z^j\}$
in $L^2(\tau)$ for $j=0,1,\dots$ to obtain a sequence of polynomials $\{q_0, q_1, \dots\}$.
Assume that $\tau$ satisfies the Bernstein-Markov property (see \cite{BS}).
That is, for all $\epsilon >0,$ there is an $M>0$ such that for all $p\in\mathbb{P}_n$ we have
\begin{equation}
\label{E:p-K}
||p||_K\leq Me^{\epsilon n}||p||_{L^2(\tau)}.
\end{equation}
Then we have the following convergence result for every $z \in \mathbb{C}$ (from \cite{BS}):
\begin{equation}
\label{ortho}
    V_K(z)=\lim_{n\rightarrow\infty}\frac{1}{2n}\log\sum_{j=0}^n|q_j(z)|^2.
\end{equation}
Furthermore, since $K$ is regular, \eqref{ortho} holds locally uniformly. It follows from \eqref{E:p-K} that given $\epsilon>0$ there is an  $M>0$ such that, for all $j\in\mathbb{N},$ we have that  
$$
||q_j(z)||_K\leq Me^{\epsilon j},
$$
and so
\begin{equation}\label{BM}
|q_j(z)|\leq Me^{n(V_K(z)+\epsilon)}, \qquad \text{ for all $j \in\{0, 1, \dots, n\} , z \in \mathbb{C}$.}
\end{equation}
We consider random variables $\xi_0, \xi_1, \dots$ satisfying
\begin{equation}
\label{meas}
\mathbb{P}(|\xi_0|>e^{|z|})=o(|z|^{-1}).
\end{equation}
The lemma below shows that the above condition is sufficient to get an upper bound in probability on $\frac{1}n \log |H_n|$. Theorem \ref{prob} shows that to prove convergence of the roots of $H_n$, it then suffices to get a pointwise lower bound on this function on a dense set.
\begin{lemma}
\label{L:prob-bd}
Let $\xi_0,\xi_1,\dots$ be i.i.d. random variables satisfying (\ref{meas}), and let $H_n(z,\omega)$ be the random polynomials given by \eqref{r1}.

For any subsequence $Y\subset\mathbb{N}$, there is a further subsequence $Y_1\subset Y$ such that almost surely, the family $\{\frac{1}{n}\log |H_n| : n \in Y_1\}$ is locally bounded above, and such that for all $z\in\mathbb{C}$,
\begin{equation}
\label{E:n-Y1}
\limsup_{n\in Y_1} \frac{1}{n}\log |H_n(z,\omega)|\leq V_K(z).
\end{equation}

\end{lemma}

\begin{proof}
It follows from (\ref{meas}) that for every $\epsilon >0$
we have
\begin{equation}
\label{E:meas-2}
\mathbb{P}(|\xi|>e^{\epsilon |z|})=o(|z|^{-1}).
\end{equation}
Letting $\Omega_{n,\epsilon}=\{\omega\in\Omega:|\xi_i(\omega)|\leq e^{\epsilon n}\quad\text{for}\quad i=0,\dots,n\},$ the asymptotics in \eqref{E:meas-2} implies that
$$
\mathbb{P}(\Omega^c_{n,\epsilon})\rightarrow 0 \qquad \text{as $n\rightarrow\infty.$}
$$
Thus, given a subsequence $Y\subset\mathbb{N}$ there is a further subsequence $Y_1\subset Y$, $Y_1= \{n_1,n_2,\dots \}$ such that
$$
\sum_{s=0}^{\infty}\mathbb{P}(\Omega^c_{n_s,\epsilon})<\infty.
$$
Therefore by the Borel-Cantelli lemma, for every $\epsilon > 0$ we can almost surely find an $s_0 \in \mathbb{N}$ such that for every $s \ge s_0$, we have
$$
|\xi_i(\omega)|\leq e^{\epsilon n_s} \qquad \text{ for every $i \in \{0, \dots n_s\}$.}
$$
Therefore for any $\epsilon > 0$, \eqref{BM} implies that almost surely,
$$
\limsup_{s \to \infty} \frac{1}{n_s} \log\left(\sum_{i=0}^{n_s}|\xi_iq_i(z)| \right) \le V_K(z) + \epsilon \qquad \text{ for every $z \in \mathbb{C}$}.
$$
Letting $\epsilon$ tend to $0$, and observing that $|H_n(\omega,z)| \le \sum_{i=0}^{n_s}|\xi_iq_i(z)|$ completes the proof of \eqref{E:n-Y1}. The fact that the sequence $\{\frac{1}{n}\log |H_n| : n \in Y_1\}$ is locally bounded above follows since the convergence in \eqref{ortho} is locally uniform.
\end{proof}

By imposing a stronger condition on the $\xi_i$, we can get an almost sure upper bound on $\frac{1}n \log|H_n|$.

\begin{lemma}
\label{L:as-bd}
Let $\xi_0,\xi_1,\dots$ be non-degenerate i.i.d. random variables satisfying 
\begin{equation}\label{upper}
\mathbb{E}\log(1+|\xi_0|)<\infty
\end{equation} 
 Then almost surely,
$$
\limsup_{n \to \infty} \frac{1}{n}\log |H_n(\omega,z)|\leq V_K(z) \quad \text{ for all $z \in \mathbb{C}$.}
$$ 
\end{lemma}
\begin{proof}
Condition (\ref{upper}) implies that
$$
\sum_{n=0}^{\infty}\mathbb{P}(|\xi_n| > e^{\epsilon n} )<\infty \qquad \text{for every } \quad \epsilon > 0.
$$
Therefore for any $\epsilon > 0$, the Borel-Cantelli lemma implies that there exists a random constant $C$ such that almost surely
$$
|\xi_n| < Ce^{\epsilon n} \qquad \text{for all $n$}.
$$
The lemma then follows by similar reasoning to that used in the proof of Lemma \ref{L:prob-bd}.
\end{proof}

We will use the following lemma, presented without proof. 

 \begin{lemma}\label{seq}
 Let $a_o,a_1,\dots.$ be a sequence of non-zero complex numbers and $A\geq 0.$
 The following equations are equivalent:
 \begin{align*}
 (i) \lim_{n\rightarrow\infty} &\frac{1}{n}\log (\max^n_{i=0}|a_i|)=A. \\
  (ii)  \lim_{n\rightarrow\infty}&\frac{1}{n}\log (\sum^n_{i=0}|a_i|)=A. \\
  (iii) \lim_{n\rightarrow\infty}&\frac{1}{2n}\log (\sum^n_{i=0}|a_i|^2)=A.
 \end{align*}
\end{lemma}
  Note if we set  $a_j=q_j(z)$ at any point $z$ where none of the polynomials $q_j(z)$ are zero, then \eqref{ortho} shows that each of the above conditions are met.

 \section{Zeros of random polynomials}
We are concerned with the zeros of random polynomials of the form (\ref{r1}). We will prove, under appropriate circumstances, the weak* convergence as $n$ approaches infinity of the normalized counting measure of the zeros  to the equilibrium measure of $K$.

Given a compact non-polar set $K \subset \mathbb{C}$,  the equilibrium measure $\mu_K$ is defined as the unique probability measure which minimizes over probability measures $\mu$ on $K$, the functional (\cite{ST}, Theorem I.3)
\begin{equation}\label{equ}
\int\int\log\frac{1}{|z-t|}d\mu(z)d\mu(t).
\end{equation}
It may also be characterized by (see  \cite{ST}, Appendix B, Lemma 2.4)
\begin{equation}\label{eqgreen}
    \mu_K=\frac{1}{2\pi}\Delta V_K^*
\end{equation}
where $\Delta$ denotes the Laplacian and the equation is in the sense of distributions.

\medskip

Now, if $p_n(z)$ is a polynomial of degree $n$, the normalized counting measure of its zeros (counting multiplicity) is given by
\begin{equation} 
\frac{1}{2\pi}\Delta \left(\frac{1}{n}\log |p_n| \right)=\frac{1}n \sum_{j=1}^n\delta (z_j),
\end{equation}
where $z_1,z_2,\dots, z_n$ are the zeros of $p_n$ and $\delta(z)$
denotes the Dirac-delta measure at $z$.

\medskip

We will use the notation $Z_{H_n}$ to denote the normalized counting measure of the zeros of the random polynomial $H_n$. That is,
\begin{equation}
Z_{H_n}=\frac{1}{2\pi}\Delta\left(\frac{1}{n}\log|H_n|\right).
\end{equation}

\begin{theorem}\label{as}
Let $H_n(z,\omega)$ be a sequence of random polynomials of the form (\ref{r1}). Suppose that the following hypotheses are satisfied:
\smallskip
\begin{enumerate}[nosep,label = (\roman*)]
\item Almost surely the sequence $\{\frac{1}{n}\log |H_n| : n \in \mathbb{N} \}$ is locally bounded above.
\item Almost surely we have
$$
\limsup_n \frac{1}{n}\log |H_n(z,\omega)|\leq V_K(z) \qquad \text{ for every $z \in \mathbb{C}$}.
$$
\item For each point $z_i$ of a   countable dense set of points $\{z_i\}_{i\in \mathbb{N}} \subset \mathbb{C}$, we  have
$$
\lim_{n \to \infty} \frac{1}{n}\log|H_n(z_i,\omega)|=V_K(z_i) \qquad \text{ almost surely}.
$$
\end{enumerate}
Then
$$
\lim_{n \to \infty} Z_{H_n}=\mu_K \qquad \text{weak* almost surely}.
$$
\end{theorem}
\begin{proof}
Almost surely in $\Omega$, we have that
 $$\lim_{n \to \infty} \frac{1}{n}\log|H_n(z_i,\omega)|=V_K(z_i)$$
 for all $z_i$. Applying Theorem \ref{b} gives that
  $$
  \lim_{n \to \infty} \frac{1}{n}\log|H_n(z,\omega)|=V_K(z),
  $$
  in $L^1_{loc}(D)$. In particular, the convergence holds in the sense of distributions. Applying $\frac{1}{2\pi}\Delta$ to both sides
  we have 
  $$\lim_{n \to \infty} Z_{H_n}=\mu_K$$ almost surely as distributions, and therefore in the weak* topology on probability measures.
\end{proof}

\begin{theorem}
\label{prob}
Let $H_n(z,\omega)$ be a sequence of random polynomials of the form (\ref{r1}). Suppose the following hypotheses are satisfied:

\smallskip

\begin{enumerate}[nosep, label=(\roman*)]
\item For any subsequence $Y\subset\mathbb{N}$
there is a further subsequence $Y_0\subset Y$ such that almost surely, the set $\{\frac{1}{n}\log |H_n| : n \in Y_0\}$ is locally bounded above, and 
$$
\limsup_{n\in Y_0} \frac{1}{n}\log |H_n(z,\omega)|\leq V_K(z) \qquad \text{ for all $z \in \mathbb{C}$}.
$$

\item For each of a countable dense set of points $\{z_i\}_{i\in \mathbb{N}} \subset \mathbb{C}$ we  have that
$$
\lim_{n \to \infty} \frac{1}{n}\log |H_n(z_i,\omega)|=V_K(z_i) \quad \text{ in probability}.
$$
\end{enumerate}
Then
$$
Z_{H_n} \xrightarrow[n \to \infty]{} \mu_K \qquad \text{ in probability }
$$
in the weak* topology on probability measures on $\mathbb{C}$. That is, for any open set $U$ in the weak* topology containing $\mu_K$, we have that
$$
\mathbb{P} \left(Z_{H_n} \in U \right) \to 1 \qquad \text{ as } n \to \infty.
$$

\end{theorem}

\begin{proof}
First we recall that (see \cite{Ka}, Lemma 3.2) a sequence of random variables with values in a metric space converges in probability to some limit if and only if every subsequence of those random variables contains a further subsequence which converges almost surely to the same limit.

\medskip

Let $H_n(z,\omega)$ be a sequence of random polynomials satisfying the hypotheses of the theorem. Let $Y\subset\mathbb{N}$ be a subsequence. We wish to find a further subsequence $Y^*\subset Y$ such that
\begin{equation}
\label{E:Z-Hn}
\lim_{n\in Y^*}Z_{H_n}=\mu_K \qquad \text{weak* almost surely.}
\end{equation}
We begin with a subsequence $Y_0\subset Y$ satisfying $(i)$.
Consider the point $z_1$. In probability,

\begin{equation}\label{pointwise}
\lim_{n\in Y_0}\frac{1}{n}\log|H_n(z_1,\omega)|=V_K(z_1).
\end{equation}
There is therefore a subsequence $Y_1\subset Y_0$ so that 
equation (\ref{pointwise}) holds almost surely.
Then we consider $z_2$. By hypothesis, equation (\ref{pointwise}) holds in probability with $Y_1$ and $z_2$ in place of $Y_0$ and $z_1$.  Therefore there is a subsequence of $Y_2\subset Y_1$ on which the equation holds at $z_2$ almost surely. We proceed successively in this way through all the points $\{z_1,z_2, \dots\}$. Using the Cantor diagonalization procedure we obtain a subsequence $Y^*\subset Y$ such that  almost surely $\{\frac{1}{n}\log|H_n| : n \in Y^*\}$ is locally bounded,
$$
\limsup_{n\in Y^*} \frac{1}{n}\log |H_n(z,\omega)|\leq V_K(z),
$$
and equation (\ref{pointwise}) holds almost surely at each point of the countable dense set $\{z_1,z_2, \dots\}$ with $Y_*$ in place of $Y_0$. Applying Theorem \ref{as} proves \eqref{E:Z-Hn}, as desired.

\end{proof}

\section{Convergence in probability}
\label{S:cvg-prob}

    We wish to use Theorem \ref{prob} to prove convergence in probability for random polynomials of the form (\ref{r1}).
    The upper bound on $\frac{1}n\log |H_n|$, hypothesis $(i)$, holds under assumption (\ref{meas}). It remains to establish a pointwise lower bound. For this, we use the Kolmogorov-Rogozin inequality. For a complex random variable $X$ and a positive real number $r$, define the concentration function
$$
\mathcal{Q}(X ; r) = \sup_{x \in \mathbb{C}} \;\mathbb{P}( X \in B(x, r)).
$$
Here $B(x, r)$ is the open ball of radius $r$ centred at $x$. 

\begin{theorem}[Kolmogorov-Rogozin Inequality, see \cite{Es}, Corollary 1 on page 304] There is a constant $C$ such that for any independent random variables $X_1, \dots, X_n$ and for any $r > 0$, we have
$$
\mathcal{Q}\left(\sum_{i=1}^n X_i ; \; r \right) \le \frac{C}{\sqrt{\sum_{i=1}^n [1 - \mathcal{Q}(X_i ; r)]}}.
$$

\end{theorem}

The concentration function $\mathcal{Q}$ also has the following elementary properties. First, rescaling a complex random variable $X$ by any $a \in \mathbb{C}\setminus \{0\}$ gives that
\begin{equation}
\label{E:Q-fact-1}
\mathcal{Q}(aX; \; r) = \mathcal{Q}\left(X; \; \frac{r}{|a|} \right).
\end{equation}
Also, if $X$ and $Y$ are independent random variables, then
\begin{equation}
\label{E:Q-fact-2}
\mathcal{Q}(X + Y; \; r) \le \mathcal{Q}(X; \; r).
\end{equation}

%
%

    Theorem \ref{convprob} uses the Kolmogorov-Rogozin inequality to establish a pointwise lower bound on functions of the form $\frac{1}n\log |\sum \xi_i a_i|$.     
     
\begin{theorem}\label{convprob}
     Let $a_0,a_1,\dots$ be a sequence of non-zero complex numbers
satisfying any of the equivalent conditions of Lemma \ref{seq}.

Let $\xi_0,\xi_1,\dots$ be a sequence of i.i.d. non-degenerate complex random variables such that (\ref{meas}) holds. Then
\begin{equation}\label{prob1}
     \lim_{n \to \infty} \frac{1}{n}\log\Big|\sum_{i=0}^n\xi_ia_i\Big| = A\end{equation}
     in probability.
          \end{theorem}
\begin{proof}
We first show that for any $\epsilon >0,$ that
\begin{equation}
\label{prob2}
\mathbb{P}\Bigg(\frac{1}{n}\log\Big|\sum_{i=0}^n\xi_ia_i\Big|>A+\epsilon\bigg)\rightarrow 0 \qquad \text{as} \;\; n \to \infty.
\end{equation}
Recall from the proof of Lemma \ref{L:prob-bd} that condition \eqref{meas} implies that
\begin{equation}
\label{E:prob-cond}
\mathbb{P}\big( |\xi_i| \le e^{\epsilon n/2} \quad \text{for} \quad i = 0, \dots, n \big) \to 1 \qquad \text{as} \;\; n \to \infty.
\end{equation}
On the event in the above probability, we have that
\begin{align*}
\frac{1}{n}\log\Big|\sum_{i=0}^n\xi_ia_i\Big| \le \frac{1}n \log \left(\max_{i=0}^n |\xi_i|\right)+ \frac{1}n \log \left( \sum_{i=0}^n |a_i| \right) \le \frac{\epsilon}2 + \frac{1}n \log \left( \sum_{i=0}^n |a_i| \right).
\end{align*}
By Lemma \ref{seq} (ii), for all large enough $n$ the right hand side above is at most $A + \epsilon$. Combining this with the convergence in \eqref{E:prob-cond} implies \eqref{prob2}.

\medskip
We now show that for any $\epsilon > 0$, that
 \begin{equation}
 \label{prob3}
\mathbb{P}\Bigg(\frac{1}{n}\log\Bigg|\sum_{i=0}^n\xi_ia_i\Big| < A-\epsilon\Bigg)\rightarrow 0 \qquad \text{as} \;\; n \to \infty.
\end{equation}
We bound the above probability in terms of the concentration function for the sum and then apply the Kolmogorov-Rogozin inequality. This gives that
\begin{align*}
\mathbb{P}\Bigg(\frac{1}{n}\log\Bigg|\sum_{i=0}^n\xi_ia_i\Bigg| < A-\epsilon \Bigg)&\leq \mathcal{Q}\Bigg(\sum_{i=0}^n\xi_ia_i;\; e^{n(A-\epsilon)}\Bigg)\\
&\leq C\Bigg(\sum_{i=0}^n(1-\mathcal{Q}(\xi_ia_i; \;e^{n(A-\epsilon)}))\Bigg)^{-1/2}.
\end{align*}
To complete the proof of (\ref{prob3}) it suffices to prove that the sum
\begin{equation}\label{prob4}
\sum_{i=0}^n(1-\mathcal{Q}(\xi_ia_i;e^{n(A-\epsilon)}))=\sum_{i=0}^n\Bigg(1-\mathcal{Q}\Bigg(\xi_i;\frac{e^{n(A-\epsilon)}}{|a_i|}\Bigg)\Bigg)
\end{equation}
approaches infinity as $n\rightarrow \infty.$ Here the equality follows from rescaling (Equation \eqref{E:Q-fact-1}).

For this, first observe that by the non-degeneracy of $\xi_0$, we can find positive numbers $D_1,D_2$ such that for all $d \le D_1$, we have that $\mathcal{Q}(\xi_0;d) \le D_2<1.$

Therefore to show that the sum in \eqref{prob4} approaches infinity as $n\rightarrow \infty$, it is enough to show that $|J_n|\rightarrow\infty$ as $n\rightarrow \infty,$ where
$$
J_n=\left\{i\leq n: \; \frac{e^{n(A-\epsilon)}}{|a_i|}\leq D_1 \right\}.
$$
 We note that in Lemma \ref{seq}, $A\geq 0. $ We will consider two cases, $A=0$ and $A>0$, and
 first  deal with the case $A=0.$ In this case, for all $i$ we have that
  $$
  \frac{e^{n(A-\epsilon)}}{|a_i|}\longrightarrow 0 \qquad \text{ as }n\rightarrow\infty.
  $$ 
 Therefore for all $i$ there exists a number $N_i$ such that $i\in J_n$ for all $n\geq N_i.$ Therefore $|J_n|\rightarrow \infty$ as  $n\rightarrow\infty.$
 
 \medskip
 
 Now, suppose $A>0$. We may assume that $0<\epsilon<A$. Fix $\delta\in (0,\epsilon /2).$ By Condition (i) of Lemma \ref{seq}, there exists an $N$ such that for all $n \ge N$, we have that 
 $$
 \quad \frac{1}{n}\log \left(\max^n_{k=0}|a_k| \right)\in [A-\delta,A+\delta].
 $$
 This implies that for all $k \ge N$, that $|a_i|\leq e^{k(A+\delta )}$ for $i \le k$.
 Also, for all $m\geq N$ there exists a $k(m)\leq m$ such that $|a_{k(m)}|\geq e^{m(A-\delta)}.$ 
These conditions on the coefficients $a_k$ guarantee that for any $k \ge N$, that
\begin{equation}\label{prob5}
|\{m\in\mathbb{N}: k=k(m)\}|\leq \frac{2\delta k}{A-\delta} + 1.
\end{equation}
Now, observe that for all $n\geq \frac{N(A - \delta)}{A - \epsilon/2},$
for every $m\in\left[\frac{n(A-\epsilon/2)}{A-\delta},n\right],$ we have that 
$$
\frac{e^{n(A-\epsilon)}}{|a_{k(m)}|}\leq e^{-\epsilon n/2}.
$$
Choosing $n$ large enough so that $e^{-\epsilon n/2}<D_1,$ we then have that

$$|J_n|\geq \Bigg|\Bigg\{ k \leq n: k= k(m) \;\; \text{ for some }\;\; m\in\Bigg[\frac {n(A-\epsilon/2)}{A-\delta},n\Bigg]\Bigg\}\Bigg|.
 $$
By \eqref{prob5}, for all $n$ large enough, the right side above can be bounded below by
$$
\frac{\frac{n(\epsilon/2-\delta)}{A-\delta}}{{1+\frac{2\delta n}{A-\delta}}}\geq\frac{(\epsilon/2-\delta)}{3\delta}.
$$
Since $\delta$ can be taken arbitrarily small, this implies that
$|J_n| \to \infty$ as $n \to \infty$, completing the proof of \eqref{prob3}.
\end{proof}

We now have all the ingredients to prove convergence in probability of the normalized counting measure of the zeros for random orthonormal polynomials.

\begin{theorem}\label{goal}
Let $K\subset\mathbb{C}$ be a regular, compact set and let $\xi_0,\xi_1,\dots$
be a sequence of non-degenerate i.i.d complex random variables satisfying (\ref{meas}). Consider the random polynomials 
$$H_n(z):=\sum_{j=0}^n\xi_iq_j(z)$$ where $\{q_j(z)\}$ are  the orthonormal polynomials with respect to a measure on $K$ satisfying the Bernstein-Markov property defined in Section \ref{S:construct}.

Then $Z_{H_n}$ converges in probability to $\mu_K$ in the weak* topology on probability measures on $\mathbb{C}.$
\end{theorem}
\begin{proof}
By (\ref{ortho}), Lemma \ref{seq}  is satisfied with $a_i=q_i(z)$
at all points $z$ where no $q_i(z)=0,$ and thus at all but countably many points in the plane. Therefore Theorem \ref{convprob} applies and we have the convergence in probability 
$$
\frac{1}{n}\log|H_n(z,\omega)| \to V_K(z)
$$
at all but at most countably many points in the plane. Hypothesis (ii) of Theorem \ref{prob} is thus satisfied. Moreover, hypothesis (i) of Theorem \ref{prob} follows from Lemma \ref{L:prob-bd}. Applying Theorem \ref{prob}, the result follows.
\end{proof}

\begin{remark}
\label{R:pritsker}
In \cite{IZ}, Ibragimov and Zaporozhets showed that the condition 
\begin{equation}
\label{E:IZ}
\mathbb{E}\log(1 + |\xi_0|) < \infty
\end{equation} 
is equivalent to the weak* almost sure convergence of $Z_{H_{n}} \to \frac{1}{2\pi} d \theta$ in the Kac ensemble case when $q_j(z) = z^i$ (see Theorem \cite{IZ}, Theorem 1). 

Motivated by this and theorems of a similar flavour in \cite{PR}, Pritsker and Ramachandran (Conjecture 2.5, \cite{PR}) asked if there exists a measure $\tau$ on the unit circle $\{z = 1\}$ and a sequence of i.i.d. random variables $\{\xi_i\}_{i \in \mathbb{N}}$ such that \eqref{E:IZ} does not hold, and such that for the basis of orthogonal polynomials $\{q_n(z)\}_{n \in \mathbb{N}}$ constructed with respect to $\tau$, a subsequence $\{Z_{H_{n_i}}\}$ of the normalized counting measures of the zeros of the polynomials 
$$
H_n(z) = \sum_{i=0}^n \xi_i q_i(z)
$$
still converges weak* almost surely to $\frac{1}{2\pi} d \theta$. 

\medskip

Theorem \ref{goal} shows that for any sequence of i.i.d. random variables $\{\xi_i\}_{i \in \mathbb{N}}$ such that 
$$
\mathbb{E}\log(1 + |\xi_0|) = \infty \qquad \text{and} \qquad \mathbb{P}(|\xi_0| > e^{|z|}) = o\left(|z|^{-1}\right), 
$$
and any sequence of orthonormal polynomials $\{q_n(z)\}_{n \in \mathbb{N}}$ constructed with respect to a Bernstein-Markov measure $\tau$ on $\{|z| = 1\}$, that $Z_{H_n} \to \frac{1}{2\pi}d\theta$ in probability, and hence any subsequence $\{Z_{H_{n_i}}\}$ has a further subsequence which converges almost surely. This resolves Pritsker and Ramachandran's conjecture.
\end{remark}

We now show that Theorem \ref{goal} is the best possible result for general orthogonal ensembles, by showing that condition \eqref{meas} is both necessary and sufficient for the Kac ensemble. To do this, we first need a lemma about random variables.
\begin{lemma}
\label{L:tail-bd} Let $X$ be a non-negative random variable. Suppose that 
\begin{equation}
\label{E:ls-cond}
\limsup_{n \to \infty} n\mathbb{P}(X > n) > 0.
\end{equation}
Then there exists a function $f:[0, \infty) \to [0, \infty)$ such that 
\begin{enumerate}[label=(\roman*)]
\item $C(f) :=\limsup_{n \to \infty} n\mathbb{P}(X > f(n)) \in (0, \infty).$ Here the limsup is taken over $n \in \mathbb{N}$, hence the use of $n$ instead of $x$.
\item For every $x, y \in [0, \infty)$, we have that
$
f(x) + y \le f(x + y).
$
\end{enumerate}

\end{lemma}

\begin{proof}
In the case when 
$
\limsup_{n \to \infty} n\mathbb{P}(X > 2n) \in (0, \infty),
$
then the function $f(x) = 2x$ works. Therefore noting that this limsup is positive by \eqref{E:ls-cond}, we can assume that 
\begin{equation}
\label{E:infinity}
\limsup_{n \to \infty} n\mathbb{P}(X > 2n) = \infty.
\end{equation}
Define a function $g: \{0, 1 \dots, \to [0, \infty)$ so that $g(0) = 0$, and for $n \ge 1$,
\begin{equation}
\label{E:reff}
\mathbb{P}(X > g(n)) \le \frac{1}n \quad \text{and} \quad \mathbb{P}(X \ge g(n)) \ge \frac{1}n.
\end{equation}
Now, for each $x \in [0, \infty)$, define
$$
f(x) = \max\{g(n) + x -n : n \in \{0, \dots, \lfloor x \rfloor\} \}.
$$
We check that $f$ satisfies the conditions of the lemma. First fix $x < y \in [0, \infty)$. For some $n \in \{0, \dots, \lfloor x \rfloor\} $, we have that 
$f(x) = g(n) + x -n$. By the definition of $f(y)$, we have that $f(y) \ge g(n) + y -n = f(x) + y - x$. Thus $f$ satisfies (i).

\medskip

Now, there must be infinitely many values of $n \in \mathbb{N}$ such that $f(n) = g(n)$. To see this, note that if there are only finitely many such values, then there exists an $m \in \mathbb{N}$ such that for all $x \ge m$, we have that $f(x) = g(m) + x - m$. Hence $g(n) \le f(n) \le 2n$ for all large enough $n$, and so by the first inequality in \eqref{E:reff},
$$
\limsup_{n \to \infty} n\mathbb{P}(X > 2n) \le  \limsup_{n \to \infty} n \mathbb{P}(X > g(n)) \le 1.
$$
This contradicts \eqref{E:infinity}. Now let $\{n_i \in \mathbb{N}\}$ be a subsequence so that $f(n_i) = g(n_i)$ for all $i$. Note that $f(n_i -1) < f(n_i) = g(n_i)$ since $f$ is strictly increasing. Therefore we have that
\begin{equation}
\label{E:ni}
\lim_{n_i \to \infty} (n_i - 1)\mathbb{P}(X > f(n_i - 1)) \ge \lim_{n_i \to \infty} (n_i - 1)\mathbb{P}(X \ge g(n_i)) \ge 1.
\end{equation}
The final inequality follows from \eqref{E:reff}. Moreover, since $f(n) \ge g(n)$ for all $n \in \mathbb{N}$, we also have that
$$
\limsup_{n \to \infty} n\mathbb{P}(X > f(n)) \le \limsup_{n \to \infty} n\mathbb{P}(X > g(n)) = 1.
$$
Here the last inequality again follows from \eqref{E:reff}. Combining this with \eqref{E:ni} implies that $C(f) = 1$.
\end{proof}

\begin{theorem}
\label{T:suff-Kac}
Consider the random polynomials 
$$H_n(z):=\sum_{j=0}^n\xi_iz^i,$$
where $\xi_0,\xi_1,\dots$
is a sequence non-degenerate i.i.d complex-valued random variables.
Then $Z_{H_n}$ converges in probability to $\frac{1}{2\pi}d\theta$ if and only if the random variables $\xi_i$ satisfy \eqref{meas}.
\end{theorem}

\begin{proof} The ``if" statement is a consequence of Theorem \ref{goal}, with $K = \{z : |z| = 1\}$ and $\tau = \frac{1}{2\pi} d\theta$ (see Example \ref{Ex:kac}). Now suppose that the random variables $\xi_i$ fail to satisfy \eqref{meas}. 

\medskip

Recall that $Z_{H_n}$ is a random variable on the space $\mathcal{M}$ of probability measures on $\mathbb{C}$ with the weak* topology.  To show that $Z_{H_n}$ does not converge in probability to  $\frac{1}{2\pi}d\theta$, it suffices to find an open set $U \subset \mathcal{M}$ that contains $\frac{1}{2\pi}d\theta$, and such that 
\begin{equation}
\label{E:meas-cvg}
\liminf_{n \to \infty} \mathbb{P}(Z_{H_n} \in U) < 1.
\end{equation}
Let 
$$
\mathcal{O} = \{ z \in \mathbb{C} : 1/2 < |z| < 3/2 \}, \quad \text{and let} \quad U = \{\mu \in \mathcal{M} : \mu(\mathcal{O}) > 1/2\}.
$$
$U$ is open in $\mathcal{M}$ by the portmanteau theorem, and contains $\frac{1}{2\pi}d\theta$.
We will show that
\begin{equation}
\label{E:sup-in-O}
\limsup_{n \to \infty} \; \mathbb{P} (Z_{H_n}(\mathcal{O}) = 0) > 0,
\end{equation}
which in turn proves \eqref{E:meas-cvg}, showing that $Z_{H_n}$ does not converge in probability to $\frac{1}{2\pi} d \theta$. To prove \eqref{E:sup-in-O}, we show that
\begin{align}
\label{E:want-poisson}
\limsup_{n \to \infty} \; \mathbb{P} \bigg( \text{There exists } &m \in \{0, \dots, n \} \text{ such that }\\
\nonumber & |\xi_m| \ge e^n|\xi_j| \text{ for all } j \in \{0, \dots, n \} , \; j \ne m \bigg) > 0.
\end{align}
To see why this is sufficient, observe that on the event in \eqref{E:want-poisson}, for all large enough $n$, we have that
$$
|\xi_m||z|^m > \sum_{\substack{ j \in [0, n] \\
j \ne m}} |\xi_j||z|^j \ge \bigg|\sum_{\substack{ j \in [0, n] \\
j \ne m}} \xi_jz^j \bigg|
$$
for all $z \in \mathcal{O}$. Therefore $H_n$ has no zeroes in $\mathcal{O}$ on this event. We now prove \eqref{E:want-poisson}. For a function $f: [0, \infty) \to [0, \infty)$ and $n \in \mathbb{N}$, define
$$
D_n(f) := n\mathbb{P}(|\xi_0| > g(n)), \quad \text{and} \quad D(f) := \limsup_{n \to \infty} D_n(f).
$$
Since \eqref{meas} is not satisfied by the $\xi_i$, we can apply Lemma \ref{L:tail-bd} to the random variable $\log |\xi_0|$ getting a function $f$ satisfying properties (i) and (ii) of that lemma. Letting $g = e^f$, we then have that
\begin{enumerate}[label=(\roman*)]
\item $D(g) \in (0, \infty)$.
\item For every $x, y \in [0, \infty)$, we have that 
$
g(x + y) \ge e^yg(x).
$
\end{enumerate}
%
%
For $\alpha \in (0, \infty)$, define $g_\alpha(x) : = g(\alpha x)$. Observe that $\alpha D(g_{\alpha}) = D(g)$.
Now define
$$
B_{n, \alpha} = \left| \{ i \le n : |\xi_i| > g_\alpha(n) \} \right|.
$$
For each $\alpha$, $B_{n, \alpha}$ is a binomial random variable with $n$ trials and mean $D_n(g_\alpha)$. Now for any $\alpha > 1$, there exists a subsequence $\{n_i\}$ such that
$$
\lim_{n_i \to \infty} \mathbb{E} B_{n_i, \alpha} = \frac{D(g)}\alpha, \quad \text{whereas} \quad \limsup_{n_i \to \infty} \mathbb{E} B_{n_i, \alpha - 1} \le \frac{D(g)}{\alpha - 1}.
$$
Therefore for large enough $\alpha$, Poisson convergence for binomial random variables implies that
$$
\limsup_{n_i \to \infty} \mathbb{P} (B_{n_i, \alpha} = 1) - \mathbb{P} (B_{n_i, \alpha - 1} \ge 2) > 0.
$$
%
%
%
By property (ii) of the function $g$, this implies \eqref{E:want-poisson}.
\end{proof}

\begin{section}{Almost sure convergence}
\label{S:almost-sure}

In this section, we consider almost sure convergence for random polynomial ensembles. Our main tool for doing this is a small ball probability theorem of Nguyen and Vu (\cite{NV}, Corollary 2.10). In \cite{NV}, this theorem stated for real-valued random variables $\xi$ satisfying the condition
$$
\mathbb{P}(1 \le |\xi_1 - \xi_2| \le C) \ge 1/2
$$
for some value of $C$, where $\xi_1, \xi_2$ are independent copies of $\xi$. However, the proof can easily be extended to all non-degenerate real random variables, which satisfy
$$
\mathbb{P}(b_1 \le |\xi_1 - \xi_2| \le b_2) > 0
$$
for some $b_1, b_2 > 0$ at the expense of changing some of the constants (this version of their theorem is stated in \cite{NV2}). The proof can also be extended to accommodate complex-valued random variables by making a few other minor modifications. 

\medskip

The result we state and use here is weaker than the result from \cite{NV}, since we don't need to use information about the arithmetic structure of the coefficient set $\mathcal{A}$.

\begin{theorem}
\label{T:small-ball} Let $0 < \epsilon < 1, \; C > 0$ be arbitrary constants, and $\beta > 0$ a parameter that may depend on $n$. Suppose that $\mathcal{A} = \{a_0, a_1 , a_2, \dots a_n\}$ is a (multi)-subset of $\mathbb{C}$ such that $\sum_{i=0}^n |a_i|^2 = 1$, and let $\xi_0, \xi_1, \dots \xi_n$ be i.i.d. non-degenerate complex random variables. Suppose additionally that 
$$
\mathcal{Q}\left(\sum_{i=0}^n \xi_i a_i \; ; \; \beta\right) \ge n^{-C}.
$$
Then there exists a constant $D$ depending only on $\xi_0$ and $\epsilon$ such that for any number $n' \in (n^\epsilon, n)$, at least $n - n'$ elements of $\mathcal{A}$ can be covered by a union of
$
\max \left( \frac{Dn^C}{\sqrt{n'}}, 1 \right)
$
balls of radius $\beta$.
\end{theorem}

We translate this into a lemma that can be applied to prove almost sure convergence of random polynomial zeros.

\begin{lemma}
\label{L:as-criterion}
 Let $\{a_{n, i} : 0 \le i \le n, \; i, n \in \{0, 1, \dots \} \}$ be a triangular array of complex numbers such that
\begin{equation}
\label{E:A-limit}
\lim_{n \to \infty} \frac{1}{2n} \log \left( \sum_{i=0}^{n} |a_{n, i}|^2 \right) = A.
\end{equation}
Let $||a^{(n)}||$ be the Euclidean norm of $(a_{0, n}, \dots , a_{n, n})$, and let $w_{n, i} = a_{n, i}/||a^{(n)}||$. Suppose that for any $\epsilon > 0$, there exists a $\delta > 0$ such that for all large enough $n$, the set 
$$
\mathcal{W}_n = \{w_{n, i} : 0 \le i \le n \}
$$
cannot be covered by a union of $n^{2/3 + \delta}$ balls of radius $e^{-\epsilon n}$. If $\{\xi_0, \xi_1, \dots \}$ is a sequence of non-degenerate i.i.d. complex random variables, then
$$
\liminf_{n \to \infty} \frac{1}n \log \left| \sum_{i=0}^n \xi_i a_{n, i} \right| \ge A \quad \text{ almost surely}.
$$

\end{lemma}

\begin{proof} For any $\epsilon > 0$, we have that
$$
\mathbb{P} \left(\frac{1}n \log \left| \sum_{i=0}^n \xi_i a_{n, i} \right| < A - 2\epsilon \right) \le \mathcal{Q}\left( \sum_{i=0}^n \xi_i a_{n, i} \; ; \; e^{n(A-2\epsilon)} \right).
$$
Therefore by the Borel-Cantelli lemma, to prove the lemma it is enough to show that for every $\epsilon > 0$, that
\begin{equation}
\label{E:ga-sum}
\sum_{i=0}^\infty \mathcal{Q}\left( \sum_{i=0}^n \xi_i a_{n, i} \; ; \; e^{n(A-2\epsilon)} \right) < \infty.
\end{equation}
For all large enough $n$, \eqref{E:A-limit} guarantees that 
$$
\left(\sum_{i=0}^n |a_{n, i}|^2 \right)^{1/2}\in [e^{n(A-\epsilon)}, e^{n(A+\epsilon)}].
$$
Therefore for such $n$, the rescaling property of $\mathcal{Q}$ (Equation \eqref{E:Q-fact-1}) implies that
\begin{equation}
\label{E:Q-comp}
\mathcal{Q}\left( \sum_{i=0}^n \xi_i a_i \; ; \; e^{n(A-2\epsilon)} \right) \le \mathcal{Q}\left( \sum_{i=0}^n \xi_i w_i \; ; \; e^{-\epsilon n} \right).
\end{equation}
Now let $\delta$ be as in the statement of the lemma for the above value of $\epsilon$. Take $n' = n^{2/3}$ in Theorem \ref{T:small-ball}. By that theorem, there exists a constant $D$ independent of $n$ such that if the right hand side of \eqref{E:Q-comp} is greater than $n^{-1 - \delta/2}$, then at least $n - n^{2/3}$ elements of $\mathcal{W}_n = (w_{0, n}, \dots w_{n, n})$ can be covered by a union of $Dn^{2/3 + \delta/2}$ balls of radius $e^{-\epsilon n}$. This implies that all elements of $\mathcal{W}_n$ can be covered by a union of $(1 + D)n^{2/3 + \delta/2}$ balls of radius $e^{-\epsilon n}$. By the assumption on the array $\{a_{n, i}\}$, this can only occur for finitely many $n$.

\medskip

Therefore the right hand side side of \eqref{E:Q-comp} is summable in $n$, and hence so is the left hand side, proving \eqref{E:ga-sum}.
\end{proof}

We can now check that certain sequences of random polynomials satisfy the conditions of Lemma \ref{L:as-criterion} for almost every value of $z$. The setting is as follows. Consider coefficients 
$$
\{f_{n, k} \in \mathbb{C} : k \in \{0, \dots n\}, n \in \{0, 1, \dots\} \}
$$
satisfying the following two conditions:
\begin{enumerate}[label = (\roman*)]
\item There exists a continuous function $V: \mathbb{C} \to \mathbb{R}$ such that for every $z \in \mathbb{C}$, we have that
\begin{equation}
\label{E:log-Cz}
\lim_{n \to \infty} \frac{1}n \log \left( \sum_{k=0}^n |f_{n, k}||z|^k \right) = V(z).
\end{equation}
Moreover, this convergence is locally uniform. We assume that $V(z)$ is subharmonic, and that $V(z) - \log(|z|)$ is bounded as $z \to \infty$. This growth condition ensures that $\frac{1}{2\pi}{\Delta V(z)}$ is a probability measure.

\item There exists a set $\mathcal{D} \subset \mathbb{C}$ whose complement has Lebesgue measure zero, such that for every $z \in \mathcal{D}$, the following holds. For any $\epsilon > 0$, there exists an $n_0(\epsilon, z) \in \mathbb{N}$ and $\delta(\epsilon, z) > 0$ such that for all $n \ge n_0(\epsilon, z)$, we have that
\begin{equation}
\label{E:f-n-big}
\left|\left\{ k \in [0, n] : |f_{n, k}||z|^k \ge e^{n(V(z) - \epsilon)} \right\} \right|\ge n^{2/3 + \delta(\epsilon, z)}.
\end{equation}
\end{enumerate} 

 \medskip
 
Now consider a sequence of i.i.d. non-degenerate complex-valued random variables $\{\xi_0, \xi_1, \dots \}$ with $\mathbb{E}\log(1 + |\xi_0|) < \infty$, and define the random polynomials
\begin{equation}
\label{E:Gn-def}
G_n(z, \omega) = \sum_{k=0}^n \xi_k f_{n, k} z^k.
\end{equation}

Then we have the following theorem. Again, we let $(\Omega, \mathcal{F}, \mathbb{P})$ be the probability space on which the random variables $\xi_i$ are defined.
\begin{theorem}
\label{T:general-KZ}
\begin{enumerate}[nosep, label = (\Roman*)]
\item For almost every $\omega \in \Omega$, the sequence $\{\frac{1}{n}\log |G_n| : n \in \mathbb{N}\}$ is locally bounded above, and
\begin{equation}
\label{E:KZ-upper}
\limsup_{n \to \infty} \frac{1}n \log |G_n(z, \omega)| \le V(z)\;\; \text{ for every $z \in \mathbb{C}$}.
\end{equation}
\item For almost every $z \in \mathbb{C}$, we have that
$$
\lim_{n \to \infty} \frac{1}n \log |G_n(z, \omega)| = V(z) \text{ \;\; for almost every $\omega \in \Omega$}.
$$

\end{enumerate}
\end{theorem}

To prove the above theorem, we need a simple lemma bounding the Lebesgue measure of the set where a polynomial can take small values. Here and throughout the remainder of this section, $\mathfrak{m}$ is Lebesgue measure in $\mathbb{C}$. 
\begin{lemma}
\label{L:poly-estimate}
Let $P_n(z)$ be a degree $n$ polynomial with leading coefficient $c$. Then for any $r > 0$,
$$
\mathfrak{m} \{ z : |P_n(z)| \le r^n \} \le \pi n r^2c^{-2/n}.
$$
\end{lemma}

\begin{proof}
Let $z_1, \dots, z_n$ be the roots of $P_n$ and let $\Delta(z_i, r)$ be the closed ball of radius $r$ around $z$. For any $r > 0$, if $z \notin \Delta(z_i, rc^{-1/n})$ for all $i  \in \{1, \dots, n\}$, then 
$
|P_n(z)| > r^n.
$
The measure of $\bigcup_{i=1}^n \Delta(z_i, r)$ is at most $\pi n r^2 c^{-2/n}$. 
\end{proof}

Note that the factor of $n$ in Lemma \ref{L:poly-estimate} can be improved upon by Cartan's estimate (see \cite{Lev}, Lecture 11). We do not need this level of precision here.

\begin{proof}[Proof of Theorem \ref{T:general-KZ}.]
Conclusion (I) holds by Lemma \ref{L:as-bd} (note that Condition \eqref{E:KZ-upper} is the same as Condition \eqref{meas}).
\medskip

%
%
\noindent We now prove (II). By (I), it is enough to show that for almost every $z \in \mathbb{C}$, that
\begin{align}
\label{E:as-low}
\liminf_{n \to \infty} \frac{1}n \log |G_n(z)|  \ge V(z) \qquad \text{ almost surely}.
\end{align}
We want to apply Lemma \ref{L:as-criterion} to prove \eqref{E:as-low}. Fix $\epsilon > 0$, let $z_0 \in \mathcal{D}$, and define
$$
J_n(z_0, \epsilon) =\left \{k \in [0, n] : e^{n(V(z_0) - \epsilon)} \le |f_{n, k}||z_0|^k \le e^{n(V(z_0) + \epsilon)} \right\}.
$$
By condition (ii) on the coefficients $f_{n, k}$, there exists a $\delta > 0$ such that for all large enough $n$, the lower bound above holds for at least $n^{2/3 + \delta}$ values of $k$. Moreover, for large enough $n$, the upper bound holds for all $k$ by condition (i). Therefore $|J_n(z_0, \epsilon)| \ge n^{2/3 + \delta}$ for large enough $n$. 

\medskip

Now, for $z \in \mathbb{C}$, let $||(f, z)_n||$ be the Euclidean norm of the vector 
$
(f_{n, 0}, f_{n, 1}z, \dots f_{n, n}z^n).
$
Let 
$$
w_{n, k}(z) = \frac{f_{n, k}z^k}{||(f, z)_n||} \quad \text{and define} \quad \mathcal{W}_n(z) = \{ w_{n, k}(z) : k \in \{0, 1, \dots n\}\}.
$$
Since $z_0 \ne 0$ (note that $0 \notin \mathcal{D}$), we can find a $\gamma > 0$ such that if $z \in B(z_0, \gamma)$, then
$$
\frac{w_{n, k}(z)}{w_{n, k}(z_0)} = \frac{|f_{n, k} z^k|}{|f_{n, k} z_0^k|} \ge e^{-\epsilon n}
$$
for any $k, n$. Therefore for large enough $n$, if $k \in J_n(z_0, \epsilon)$ and $z \in B(z_0, \gamma)$, then $|w_{n, k}(z)| \ge e^{-3\epsilon n}.$ 

\medskip

 Now suppose that $z \in B(z_0, \gamma)$ is such that $\mathcal{W}_{n}(z)$ can be covered by $n^{2/3 + \delta/2}$ balls of radius $e^{-7\epsilon n}$. As long as $n$ is large enough, there must exist $k_1 < k_2 \in J_n(z_0, \epsilon)$ such that $|k_1 - k_2| \le n^{1 - \delta/2}$ and $|w_{n, k_1}(z) - w_{n, k_2}(z)| \le e^{-6 \epsilon n}$. We can write
\begin{equation}
\label{E:w-k-1-2}
|w_{n, k_1}(z) - w_{n, k_2}(z)| = |w_{n, k_1}(z)|\left|\frac{f_{n, k_2}}{f_{n, k_1}}z^{k_2 - k_1} - 1 \right| \ge e^{-3 \epsilon n}\left|\frac{f_{n, k_2}}{f_{n, k_1}}z^{k_2 - k_1} - 1\right|.
\end{equation}
Since both $k_1, k_2 \in J_n(z_0, \epsilon)$, we have that $|f_{n, k_2}/f_{n, k_1}| \ge e^{-2\epsilon n}|z_0|^{k_1 - k_2}$. Using this we can apply Lemma \ref{L:poly-estimate} to \eqref{E:w-k-1-2} get that 
$$
\mathfrak{m} \Big\{ z \in B(z_0, \gamma) : |w_{n, k_1}(z) - w_{n, k_2}(z)| < e^{-6 \epsilon n} \Big\} \le \pi n |z_0|^2 e^{-10\epsilon n^{\delta/2}}.
$$
Therefore by a union bound, 
\begin{equation*}
L_n := \mathfrak{m} \bigg\{ z \in B(z_0, \gamma) : \text{ $\mathcal{W}_{n}(z)$ can be covered by $n^{2/3 + \delta/2}$ balls of radius $e^{-7\epsilon n}$} \bigg\}
\end{equation*}
is at most $n^3 \pi |z_0|^2 e^{-10\epsilon n^{\delta/2}}$ for all large enough $n$. The sequence $L_n$ is summable in $n$, so by the Borel-Cantelli lemma, there exists a $\delta > 0$ such that for almost every $z \in B(z_0, \gamma)$, the set $\mathcal{W}_{n}(z)$ can be covered by $n^{2/3 + \delta}$ balls of radius $e^{-7\epsilon n}$ for at most finitely many $n$. 

\medskip

This holds for every $z_0 \in \mathcal{D}$ for some $\gamma$ and $\delta$. Therefore we can extend this result to get that for almost every $z \in \mathbb{C}$, there exists a $\delta > 0$ such that the set $\mathcal{W}_{n}(z)$ can be covered by $n^{2/3 + \delta}$ balls of radius $e^{-7\epsilon n}$ for at most finitely many $n$.

\medskip

 Since $\epsilon > 0$ was arbitrary, this implies that for almost every $z \in \mathbb{C}$, for every $\epsilon > 0$ there exists a $\delta > 0$ such that the set $\mathcal{W}_{n}(z)$ can be covered by $n^{2/3 + \delta}$ balls of radius $e^{-7\epsilon n}$ for at most finitely many $n$. 
By Lemma \ref{L:as-criterion}, this implies \eqref{E:as-low} for almost every $z \in \mathbb{C}$.
\end{proof}

We can now use Theorem \ref{T:general-KZ} to prove almost sure convergence of the normalized counting measure of the zeros for $G_n$.

\begin{theorem}
\label{T:general-KZ-roots}
Let $G_n(z, \omega)$ be as in \eqref{E:Gn-def}, where the coefficients $f_{n, k}$ satisfy conditions (i) and (ii). Let $Z_{G_n}(\omega)$ be the normalized counting measure of the zeros of $G_n$. Then for almost every $\omega \in \Omega$, we have that
$$
Z_{G_n}(\omega) \to \frac{1}{2\pi} \Delta V(z) \qquad \text{ in the weak* topology.}
$$
\end{theorem}

\begin{proof}
This follows immediately from Theorem \ref{T:general-KZ} and Theorem \ref{as} (note that the limit $V(z)$ in Theorem \ref{T:general-KZ} is subharmonic).
\end{proof}

\begin{corollary}
\label{T:KZ-prob}
Let $G_n(z,\omega)$ be as in \eqref{E:Gn-def}, where the coefficients $f_{n, k}$ satisfy conditions (i) and (ii). Suppose the $\xi_k$
are non-degenerate i.i.d. complex valued random variables satisfying (\ref{meas}). Then
$$Z_{G_n}(\omega)\rightarrow \frac{1}{2\pi}\Delta V(z)$$  in probability in the weak* topology.
\end{corollary}
\begin{proof}
We need to check the conditions of Theorem \ref{prob}. Condition (i) follows by using the reasoning of Lemma \ref{L:prob-bd}. Condition (ii) then follows from the lower bound \eqref{E:as-low} shown in the proof of Theorem \ref{T:general-KZ} (note that this bound only requires the non-degeneracy of the $\xi_k$s).
%
%
%
%
%
\end{proof} 
 \bigskip
 \noindent\textbf{Special Cases of Theorem \ref{T:general-KZ-roots}}. \qquad  We can consider the following types of coefficients, first considered by Kabluchko and Zaporozhets in \cite{KZ}.

\medskip

Assume that there is a function $f: [0, 1] \to [0, \infty)$ satisfying the following conditions:
\smallskip

\begin{enumerate}[nosep, label=(\roman*)]
\item $f(t)$ is positive and continuous for all $t$.
\item{$\lim_{n \to \infty} \sup_{k \in [0, n]} \left| |f_{n, k}|^{1/n} - f(k/n)\right| = 0.$}
\end{enumerate}

\medskip

It is easy to check that if the coefficients $f_{n, k}$ satisfy the above properties, then they satisfy the conditions required for Theorem \ref{T:general-KZ}. This gives us the following corollary.

\begin{corollary}
\label{T:KZ-as}
Let $G_n(z) = \sum_{k=0}^n \xi_k f_{n, k} z^k$ be the random polynomial with coefficients $f_{n, k}$ above, where the $\xi_i$s are non-degenerate i.i.d. complex random variables satisfying $\mathbb{E}\log(1 + |\xi_0|) < \infty$. Let $Z_{G_n}$ be the normalized counting measure of the zeros of $G_n$. For each $z$, the function $V(z)$ of \eqref{E:log-Cz} is given by
$$
V(z) = \sup_{t \in [0, 1]} \log |z|^t f(t).
$$
Then for almost every $\omega \in \Omega$, we have that
$$
Z_{G_n}(\omega) \to \frac{1}{2\pi} \Delta V(z) \qquad \text{ in the weak* topology.}
$$
\end{corollary}

%
%

We can also use Theorem \ref{T:general-KZ-roots} to look at the roots of certain random orthogonal polynomial ensembles. Since Theorem \ref{T:general-KZ-roots} allows for an array of coefficients $\{f_{n, k}\}$ rather than just a sequence $\{f_n\}$, we can consider orthogonal polynomials with respect to both a weight function and a measure. Random polynomial ensembles of this form have been previously studied in \cite{Ba1}, \cite{Ba2}, and \cite{BL}.  We give an example where condition (ii) on the coefficients $f_{n, k}$ can be directly verified.

\medskip

Let $K \subset \mathbb{C}$ be compact, and let $S: K \to \mathbb{R}$ be a real-valued continuous function. Define the weighted Green function
\begin{align*}
V_{K, S}(z) &= \text{sup} \left \{ \frac{1}{\text{deg}(p)} \log |p(z)| \; \Big| \; p \in \mathcal{P}_n \;\; \text{and} \;\; ||pe^{-nS}||_K \le 1 \right \} \\
&= \sup \left \{ u \Big| \; u \in \mathcal{L}(\mathbb{C}), u \le S \; \text{ on }\;  K \right\}.
\end{align*}
We denote by $V^*_{K, S}$ the upper semicontinuous regularization of $V_{K, S}$. The distribution $\frac{1}{2\pi} \Delta V^*_{K, S}$ is a probability measure on $K$.

\medskip

We say that $K$ is \textit{locally regular} if for all $a \in K$ and $r > 0$, the function $V_{K \cap \Delta(a, r)}$ is continuous at $a$, where $\Delta(a, r)$ denotes the closed disk centred at $a$ with radius $r$. If $K$ is locally regular and $S$ is continuous, then $V_{K, S}$ is continuous, and so $V_{K, S} = V^*_{K, S}$.

\medskip

Let $\tau$ be a finite, positive, Borel measure on $K$. We say that $\tau$ satisfies the strong Bernstein-Markov property if for all $S$ continuous and $\epsilon > 0$ there is a constant $C = C(S, \epsilon)$ such that for every $n$, we have that
\begin{equation}
\label{E:BM-strong}
||pe^{-nS}||_K \le Ce^{\epsilon n} ||p||_{L^2(e^{-2nS}\tau)}
\end{equation}
for all $p \in \mathcal{P}_n$.

Now consider a locally regular non-polar compact set $K$, a continuous function $S$, and a finite measure $\tau$ on $K$ satisfying the strong Bernstein-Markov property. For each $n$, we define orthonormal polynomials $\{q_0^{(n)}, \dots, q_n^{(n)}\}$ by applying the Gram-Schmidt procedure to the monomials $\{1, z, \dots, z^n\}$ in $L^2(e^{-2nS}\tau)$. For every $z \in \mathbb{C}$, we have that
\begin{equation}
\label{E:bergman-weight}
V_{K, S}(z) = \lim_{n \to \infty} \frac{1}{2n} \log \left( \sum_{j=0}^n |q_j^{(n)}(z)|^2 \right),
\end{equation}
where the convergence is uniform on compact subsets of $\mathbb{C}$ (see \cite{BL}). 

\medskip

We say that $K$ is circularly symmetric if for every $z \in \mathbb{C}$, we have that $z \in K$ if and only $|z| \in K$. $S$ is circularly symmetric if $S(z) = S(|z|)$, and $\tau$ is circularly symmetric if for any rotation $R$, the pushforward measure $R_*\tau$ is equal to $\tau$. 

\medskip

If $K, \tau$, and $S$ are all circularly symmetric, then the polynomials $q_j^{(n)}$ are of the form $f_{n, j}z^j$. In this case we can apply Theorem \ref{T:general-KZ-roots} to random polynomials formed from the set $\{q_j^{(n)}\}$.

\medskip

\begin{corollary}
\label{C:circ-sym}
Suppose that $K \subset \mathbb{C}$ is a locally regular non-polar compact set, $S$ is a continuous function on $K$, and $\tau$ is a finite measure on $K$ satisfying the strong Bernstein-Markov property. Suppose additionally that $K, \tau$, and $S$ are all circularly symmetric, and let 
$$
\{q_j^{(n)} : j \in \{0, 1, \dots, n \}, n \in \mathbb{N} \}
$$
be as constructed above. Define
\begin{equation}
\label{E:Hn}
G_n(z, \omega) = \sum_{j=0}^n \xi_j q_j^{(n)}(z),
\end{equation}
where $\{\xi_0, \xi_1, \dots \}$ is a sequence of non-degenerate i.i.d complex random variables such that $\mathbb{E} \log (1 + |\xi_0|) < \infty$. Let $Z_{G_n}$ be the normalized counting measure of the zero of $G_n$.
Then for almost every $\omega \in \Omega$, we have that
$$
Z_{G_n} \to \frac{1}{2\pi} \Delta V_{K, S} \qquad \text{in the weak* topology of probability measures on $\mathbb{C}$.}
$$
\end{corollary}

\begin{proof}
As mentioned above, each of the polynomials $q_j^{(n)}(z)$ is of the form $f_{n, j} z^j$ for some real number $f_{n, j}$. The coefficients $f_{n, j}$ satisfy condition (i) on the coefficients in Theorem \ref{T:general-KZ} by \eqref{E:bergman-weight}.

\medskip

We now show that the weights $f_{n, j}$ satisfy condition (ii). For any $j$, we have that
\begin{align*}
f_{n, j} = \left( \int_K |z|^{2j}e^{-2nS(z)}d\tau(z) \right)^{-1/2}.
\end{align*}
By this formula, it is easy to see that
$$
f_{n, 0}^{1/n} \to \inf_{z \in \text{supp}(\tau)} e^{S(z)} \quad \text{and} \quad f_{n, 1}^{1/n} \to \inf_{z \in \text{supp}(\tau) \setminus \{0\}} e^{S(z)} \qquad \text{as } n \to\infty.
$$
Therefore for any $\epsilon > 0$, there exists an $n \in \mathbb{N}$ such that for all $n \ge N$, we have that
\begin{equation*}
f_{n, 0} \le e^{\epsilon n} f_{n, 1},
\end{equation*}
so for any fixed $z \in \mathbb{C} \setminus \{0\}$ and $\epsilon > 0$,  for all large enough $n$ we have that
\begin{equation}
\label{E:ff}
\max \{f_{n, j}|z|^j : j \in \{0, \dots, n\} \} \le e^{\epsilon n} \max \{f_{n, j}|z|^j : j \in \{1, \dots, n\} \}.  
\end{equation}
Therefore to prove condition (ii) it is enough to show that for any $\epsilon > 0$, there exists an $N$ such that for all $n \ge N$, if $j_1, j_2 \in \{1, \dots, n\}$ and $|j_1 - j_2| \le n^{3/4}$, then
\begin{equation}
\label{E:want-f-compare}
\frac{f_{n, j_2}}{f_{n, j_1}} \ge e^{-\epsilon n}.
\end{equation}
Combined with \eqref{E:ff}, this shows that for any $z \in \mathbb{C} \setminus \{0\}$, for all large enough $n$, at least $n^{3/4}$ values of $f_{n, j}|z|^j$ are close to the maximum value, which must itself be close to $e^{nV(z)}$ by \eqref{E:bergman-weight}. This gives condition (ii).
%
%
%
%
To prove \eqref{E:want-f-compare}, fix $\epsilon > 0$, and choose $\delta > 0$ small enough so that
$$
\max_{z, w \in B(0, \delta)} |S(z) - S(w)| \le \frac{\epsilon}4.
$$
Then for any  $n, j_1, j_2$, we have that
\begin{align}
\label{E:M-def}
e^{-\epsilon n} M &\le \frac{f^2_{n, j_2}}{f^2_{n, j_1}}, \qquad \text{ where} \\
\nonumber M &= \frac{ \int_{B(0, \delta)} |z|^{2j_1}e^{-2nS(0)} d \tau(z)+ \int_{K \setminus B(0, \delta)} |z|^{2j_1} e^{-2 n S(z)} d \tau(z)}{\int_{B(0, \delta)} |z|^{2j_2} e^{-2nS(0)}d \tau(z) + \int_{K \setminus B(0, \delta)} |z|^{2j_2} e^{-2 n S(z)} d \tau(z)}.
\end{align}
Now let $R > 1$ be large enough so that $K \subset B(0, R)$, and such that $R \ge \delta^{-1}$. For all large enough $n$, whenever $|j_1 - j_2| \le n^{3/4}$, we have that
\begin{equation}
\label{E:contain-1}
\frac{\int_{K \setminus B(0, \delta)} |z|^{2j_1} e^{-2 n S(z)} d \tau(z)}{\int_{K \setminus B(0, \delta)} |z|^{2j_2} e^{-2 n S(z)} d \tau(z)} \ge R^{-|j_1 - j_2|} \ge e^{-\epsilon n}.
\end{equation}
If $\tau(B(0, \delta) \setminus \{0\}) = 0$, then this proves \eqref{E:want-f-compare}. If not, then there exists a $\gamma > 0$ such that for every $j \in \{1, 2, \dots\}$, we have that
$$
\int_{B(0, \delta)} |z|^j d\tau(z) \le 2 \int_{B(0, \delta) \setminus B(0, \gamma)} |z|^j d\tau(z).
$$
Therefore as in \eqref{E:contain-1}, for all large enough $n$, whenever $|j_1 - j_2| \le n^{3/4}$, we have that
\begin{equation}
\label{E:contain-2}
\frac{ \int_{B(0, \delta)} |z|^{2j_1} e^{-2nS(0)} d \tau(z)}{\int_{B(0, \delta)} |z|^{2j_2} e^{-2nS(0)} d \tau(z)} \ge \frac{\int_{B(0, \delta)} |z|^{2j_1}d \tau(z)}{2\int_{B(0, \delta)\setminus B(0, \gamma)} |z|^{2j_2} d \tau(z)} \ge e^{-\epsilon n}.
\end{equation}
Combining this with \eqref{E:M-def} implies \eqref{E:want-f-compare}.
\end{proof}

\end{section}

\section{Multivariable Case}

In this section, we extend the results of Sections \ref{S:cvg-prob} and \ref{S:almost-sure} to the multivariable case. We present the both the preliminaries and theorems in the same order as they were presented in the one-variable case.
\subsection{Preliminaries}
\label{SS:prelim-multi}
In this subsection we will list some results of pluripotential theory.

\medskip

Let $D\subset\mathbb{C}^d$ be an open set. A function $u$ on $D$ is \textit{plurisubharmonic} if it
\begin{enumerate}[nosep, label = (\roman*)]
\smallskip
\item takes values in $[-\infty,+\infty)$.
\item is upper semicontinuous.
\item for each $a\in D$ and $b\in\mathbb{C}^d,$ the function
$\lambda\rightarrow u(a+\lambda b)$ is subharmonic on the set $\{\lambda\in \mathbb{C} :a+\lambda b\in D\}.$
\end{enumerate}

\medskip

We denote by $psh(D)$ the collection of plurisubharmonic functions on $D$. If $f$ is analytic on $D$ then $\log |f|\in psh(D)$.

\medskip

We say that a set $E\subset \mathbb{C}^d$ is \textit{pluripolar} if there is a non-constant plurisubharmonic function $u$ with $E\subset \{u=-\infty\}$. Plurisubharmonic functions are locally integrable  and thus pluripolar sets are again of Lebesgue measure zero.

\medskip

Let $\mathcal{P}_n$ denote the space of polynomials of (total) degree $\leq n$ (in $z=(z_1,...,z_d)$). Let
$$
\mathcal{L}(\mathbb{C}^d)=\Big \{u\in sh(\mathbb{C}^d)\;|\;\; u(z)-\log |z| \;\;\text {is bounded as}\;\; |z| \rightarrow +\infty \Big\}.
$$
If $p\in\mathcal{P}_n$, then $\frac{1}{deg(p)}\log|p|\in\mathcal{L}(\mathbb{C}^d).$ 

\smallskip

\noindent For $K\subset\mathbb{C}^d$ the Green function  of $K$ is given by
$$
V_K(z):=\sup \left \{\frac{1}{deg(p)}\log |p(z)|\;\;\Big| \;\;p \;\;
\text{is a non-constant polynomial, and}\;\; ||p||_K \leq 1\right\}.
$$
When $K$ is non-pluripolar, then $V_K^*\in \mathcal{L}(\mathbb{C}).$ $K$ is \textit{regular} when $V_K$ is continuous, i.e. $V_K=V_K^*.$ 
Again we will restrict to regular sets.
\begin{Example}
Let $K=\{(z_1,...z_d):|z_i|=1 \;\;\text{for}\;\; i=1,\dots,d\}.$ Then 
$$
V_K(z_1, \dots, z_d) =\max(0,\log|z_1|,...,\log|z_d|\}.
$$
\end{Example}

We refer to $\mathbb{N}^d $ as the set of multi-indices. For
$\alpha\in\mathbb{N}^d$ set 
$$
|\alpha|:=\alpha_1+...+\alpha_d.
$$ 
We consider the following lexicographic ordering on $\mathbb{N}^d$. For two multi-indices $\alpha, \beta$, we say that $\alpha > \beta$ if either $|\alpha| > |\beta|$, or if 
$
|\alpha| = |\beta|$ and there exists $l\in\{1,\dots,d\}$ such that $\alpha_l>\beta_l$ and $\alpha_k=\beta_k$ for $k=1,\dots,l-1.$

\medskip

The monomials will be denoted by $e_{\alpha}(z)=z^{\alpha}=z_1^{\alpha_1}z_2^{\alpha_2} \dots z_d^{\alpha_d}.$  For each $\alpha \in \mathbb{N}^d$ we define
$$
\mathcal{P}_{\alpha}:= \left\{\sum_{\beta\leq\alpha}c_{\beta}e_{\beta}(z): c_{\beta}\in\mathbb{C}\quad \text{and}\quad c_{\alpha}\neq 0 \right\}.
$$

\subsection{Construction of Random Polynomials}
\label{SS:prelim-multi}

We will consider random polynomials of the form
\begin{equation}
\label{E:Hn-multi}
H_n(z,\omega)=\sum_{|\alpha|\leq n}\xi_{\alpha}q_{\alpha}
\end{equation}
where the $\xi_{\alpha}$ are i.i.d complex-valued random variables and the $q_{\alpha} \in \mathcal{P}_\alpha$ are orthonormal polynomials as constructed below.
 
 \medskip
 
 Let $\tau$ be a finite measure on a non-polar compact set $K \subset \mathbb{C}^d$. For each multi-index $\alpha$ apply the Gram-Schmidt orthogonalization procedure to the monomials $\{e_{\beta}(z):\beta\leq\alpha\}$
in $L^2(\tau)$. Here we use the lexicographic ordering introduced earlier. This gives an array of orthonormal polynomials $\{q_{\alpha}(z)\}$ indexed on the multi-indices. 

\medskip

Assume that $\tau$ satisfies the Bernstein-Markov property. Then we have the following (see \cite{BS}):
\begin{equation*}
    V_K(z)=\lim_{n\rightarrow\infty}\frac{1}{2n}\log\left(\sum_{\alpha\leq n}|q_{\alpha}(z)|^2\right).
\end{equation*}
For $K$ regular, this convergence is locally uniform. It also follows from the Bernstein-Markov property that given $\epsilon>0$ there is an $M>0$ such that for all multi-indices $\alpha$, we have 
\begin{equation*}
|q_{\alpha}(z)|\leq Me^{|\alpha|(V_K(z)+\epsilon)}.
\end{equation*}
Let $\xi_{\alpha}$ be i.i.d non-degenerate complex random variables satisfying 
\begin{equation}
\label{multimeas}
\mathbb{P}(|\xi|>e^{|z|})=o(|z|^{-d}).
\end{equation}
Then we have the following lemmas, which can be proven in the same way as Lemma \ref{L:prob-bd} and Lemma \ref{L:as-bd}.
\begin{lemma}
\label{L:upper-multi}
Let $\{\xi_{\alpha} : \alpha \in \mathbb{N}^d\}$ be non-degenerate i.i.d. complex random variables satisfying (\ref{multimeas}). Let $H_n(z, \omega)$ be as in \eqref{E:Hn-multi}.

Then for every subsequence $Y\subset \mathbb{N}$ there is a further subsequence $Y_1\subset Y$ such that almost surely, $\{\frac{1}{n}\log |H_n| : n \in Y_1\}$ is locally bounded above and
$$
\limsup_{n\in Y_1}\frac{1}{n}\log | H_n(z, \omega) |\leq V_K(z).
$$ 
\end{lemma}

\begin{lemma}
\label{L:multi-as-upper}
Let $\{\xi_{\alpha} : \alpha \in \mathbb{N}^d\}$ be non-degenerate i.i.d. random variables such that
\begin{equation}
\label{E:multi-as-bd}
\mathbb{E} [\log( 1+ |\xi_\alpha|)]^d < \infty.
\end{equation}
Then almost surely, $\{\frac{1}{n}\log |H_n| : n \in \mathbb{N} \}$ is locally bounded above and
$$
\limsup_{n\to \infty }\frac{1}{n}\log|H_n(z, \omega)|\leq V_K(z).
$$ 
\end{lemma}

\subsection{Zeros of Polynomials}
In more than one variable, the zeros of a (non-constant) polynomial are not a bounded set. The zero set is, however, pluripolar and so of Lebesgue measure zero.

\medskip

The zero set of a polynomial $p$ is viewed as a positive (1,1) current (see \cite{K}) given by:
  $$
  Z_p:=dd^c\Big(\frac{1}{deg(p)}\log|p|\Big),
  $$ 
  where $dd^c=\frac{i}{\pi}\partial\overline{\partial}.$ When paired with a test form $\phi$ (i.e. a smooth $(n-1,n-1)$ form with compact support)
$\langle Z_p,\phi\rangle\in\mathbb{C}.$ 

For a random polynomial $H_n$ we consider its zero set as a  $(1,1)$ positive current $Z_{H_n}:=dd^c(\frac{1}{n}\log |H_n(z)|).$ 
  
\subsection{Convergence in Probability}

We state, without proof the following lemma, analogous to Lemma \ref{seq}.
\begin{lemma}\label{seqmulti}
Let $\{a_{\alpha}\}_{\alpha\in\mathbb{N}^d}$ be an array of non-zero complex numbers and $A\geq 0$ a real number.
The following equations are equivalent:
\begin{align*}
\text{(i)}\quad\lim_{n\rightarrow \infty}&\frac{1}{n}\log (\max_{|\alpha|\leq n} |a_\alpha|)=A.\\
\text{(ii)}\quad \lim_{n\rightarrow \infty}&\frac{1}{n}\log(\sum_{|\alpha|\leq n}|a_{\alpha}|)=A\\
\text{(iii)}\quad\lim_{n\rightarrow \infty}&\frac{1}{2n}\log(\sum_{|\alpha|\leq n}|a_{\alpha}|^2)=A.
\end{align*}
\end{lemma}

The following theorem and its proof are similar to Theorem \ref{convprob}.

\begin{theorem}\label{convprobmulti}
     Let $\{a_{\alpha}\}_{\alpha\in\mathbb{N}^d}$ be an array of non-zero complex numbers
satisfying the equivalent conditions of Lemma \ref{seqmulti}.
Let $\{\xi_{\alpha}\}_{\alpha\in\mathbb{N}^d}$ be an array of i.i.d. non-degenerate random variables such that (\ref{multimeas}) holds. Then
\begin{equation}\label{prob1multi}
     \frac{1}{n}\log\Big|\sum_{|\alpha|\leq n}\xi_{\alpha}a_{\alpha}\Big|\rightarrow A \qquad \text{ in probability}.
     \end{equation}
     \end{theorem}
This gives a multivariable analogue of Theorem \ref{goal}.
\begin{theorem}\label{goalmulti}
Let $K\subset\mathbb{C}^d$ be a regular non-pluripolar compact set and let $\{\xi_{\alpha}\}_{\alpha\in\mathbb{N}^d}$ be non-degenerate
i.i.d. complex random variables satisfying (\ref{multimeas}). Consider the random polynomials
$$
H_n(z,\omega)=\sum_{|\alpha|\leq n}\xi_{\alpha}q_{\alpha}(z).
$$
    Then $\frac{1}{n}\log|H_n|\rightarrow V_K$ in probability in $L^1_{loc}(\mathbb{C}^d)$
    and $Z_{H_n}\rightarrow dd^cV_K$ in probability in the sense that for every subsequence $Y\subset\mathbb{N}$ there is a further subsequence $Y^*\subset Y$ such that 
    $$
    \lim_{n\in Y^*} Z_{H_n}=dd^cV_K \qquad \text{almost surely.}
    $$
     
\end{theorem}
\begin{proof}
The multivariable versions of Theorems \ref{a} and \ref{b}, with plurisubharmonic and pluripolar replacing subharmonic and polar are valid (see \cite{K} and  \cite{BL}). The multivariable version of Theorem \ref{prob} is similarly valid. With $a_{\alpha}=q_{\alpha}(z)$,  Lemma \ref{seqmulti} holds at all points of $\mathbb{C}^d$ except for the set of Lebesgue measure zero where at least one of the polynomials $q_{\alpha}(z)=0$. Therefore Theorem \ref{goalmulti} follows from Lemma \ref{L:upper-multi} and Theorem \ref{convprobmulti} as in the one variable case, \textit{mutatis mutandis}.
\end{proof}

\subsection{Almost Sure Convergence}

In this subsection, we give a proof of almost sure convergence of the roots of the Kac ensemble in dimension $d = 2$. When $d \ge 3$, we can actually use the Kolmogorov-Rogozin inequality to get a strong enough bound to apply the Borel-Cantelli lemma, since the number of random variables in the Kac ensemble is $O(n^d)$, so the Kolmogorov-Rogozin inequality gives (at best) a small ball probability bound of order $O(n^{-d/2})$. This approach is taken in \cite{Ba2}.

\medskip

Let $\{\xi_{i, j} : (i, j) \in \mathbb{N}^2 \}$ be an array of i.i.d. random variables. Define the two-variable Kac ensemble by
\begin{equation}
\label{E:2-Kac}
K_n(z_1, z_2) = \sum_{0 \le i + j \le n} \xi_{i, j} z_1^i z_2^j.
\end{equation}

\begin{theorem}
\label{T:2-Kac} Let $K_n$ be as in \eqref{E:2-Kac}, and suppose that the random variables $\xi_i$ satisfy $\mathbb{E} [\log(1 + |\xi_0|)]^2 < \infty.$ Let $Z_{K_n}$ be the zero set of $K_n$. 
Then
$$
\lim_{n \to \infty} Z_{K_n} = dd^c V \quad \text{ almost surely},
$$
where $V(z, w) = \max(0, \log(z), \log(w)).$
\end{theorem}

\begin{proof} A multivariable version of Theorem \ref{as} holds with the same conditions, which we check here.
 The fact that almost surely,
$\{\frac{1}{n}\log |K_n| : n \in \mathbb{N}\}$ is locally bounded above and
\begin{equation}
\label{E:multi-upper}
\limsup_{n \to \infty} \frac{1}n \log |K_n(z)| \le V(z)\;\; \text{ for every $z \in \mathbb{C}^2$},
\end{equation}
follows from Lemma \ref{L:multi-as-upper}. We now show that for every $(z_1, z_2) \in \mathbb{C}^2$ with $z \ne 0$, that
\begin{equation}
\label{E:K_n-want}
\liminf_{n \to \infty} \frac{1}n \log |K_n(z_1, z_2)| \ge V(z_1, z_2)\quad \text{ almost surely}.
\end{equation}
We either have that 
$$
\lim_{n \to \infty} \frac{1}n \log \left(\sum_{0 \le i \le n} |z_1|^i \right) = V(z_1, z_2), \quad \text{ or } \quad \lim_{n \to \infty} \frac{1}n \log \left(\sum_{0 \le i \le n} |z_2|^i \right) = V(z_1, z_2).
$$
We will assume that the first option occurs. Then by the proof of Theorem \ref{T:general-KZ}, for any $\epsilon > 0$, we have that
\begin{equation}
\label{E:Q-multi}
\sum_{n=0}^\infty \mathcal{Q}\left(\sum_{0 \le i \le n} \xi_{i, 0} z_1^i  \; ; \; e^{n(V(z_1, z_2) - \epsilon)} \right)  < \infty.
\end{equation}
By property \eqref{E:Q-fact-2} for the concentration function $\mathcal{Q}$ regarding sums of independent random variables, for every $\epsilon > 0$ we have that
$$
\sum_{n=0}^\infty \mathcal{Q}\left(K_n(z)  \; ; \; e^{n(V(z_1, z_2) - \epsilon)} \right)  < \infty,
$$
which in turn proves \eqref{E:K_n-want}.
\end{proof}

\begin{remark} Though we have not done so here, it is possible to prove a much more general result in the multivariable case with random polynomial ensembles of the form
$$
\sum_{|\alpha| \le n} \xi_\alpha f_{n, \alpha} z^\alpha,
$$
for some coefficients $f_{n, \alpha}$ satisfying similar conditions to those specified in Section \ref{S:almost-sure}. All of the necessary results used in Section \ref{S:almost-sure} have corresponding multivariable versions, and the proof goes through in a similar way.
\end{remark}

\end{document}